\documentclass{amsart}
\usepackage{amssymb,amsmath,amsthm,hyperref}

\usepackage[normalem]{ulem}

\theoremstyle{theorem}
\newtheorem{theorem}{Theorem}
\newtheorem{proposition}[theorem]{Proposition}
\newtheorem{corollary}[theorem]{Corollary}

\theoremstyle{definition}
\newtheorem{definition}{Definition}

\title[{Combinatorial Extensions on Almost-Distinct Partitions}]{Combinatorial Extensions on Andrews and El~Bachraoui's Almost-Distinct Partitions}

\author{Brian Hopkins}
\address{Department of Mathematics and Statistics, 
Saint Peter's University, Jersey City, NJ 07306, USA}
\email{bhopkins@saintpeters.edu}

\date{}

\begin{document}

\maketitle

\begin{abstract}
Andrews and El Bachraoui recently studied integer partitions where the smallest part is repeated a specified number of times and any other parts are distinct.  Their results included two ``surprising identities'' for which they requested combinatorial proofs.  We provide combinatorial proofs for an infinite family of related identities and also consider the analogous partitions where only the largest part can be repeated.  The results give rise to infinite families of inequalities for the number of partitions with distinct parts.
\end{abstract}

\section{Background}

Given an integer $n > 0$, we say $\lambda = (\lambda_1, \dots, \lambda_t)$ is a partition of $n$ if each $\lambda_i$ is a positive integer, $\sum \lambda_i = n$, and $\lambda_1 \ge \lambda_2 \ge \cdots \ge \lambda_t$.
Write $P(n)$ for the set of partitions of $n$.
Requiring that the inequalities on the parts are strict, i.e., $\lambda_1 > \lambda_2 > \cdots > \lambda_t$, gives $Q(n) \subset P(n)$, the set of distinct part or strict partitions of $n$.  Throughout, we use lowercase letters for set sizes, e.g., $q(n) = |Q(n)|$.

In 2008, George Andrews initiated a rich subfield in integer partition theory focusing on the smallest parts of partitions \cite{A08}.  In this vein, Andrews and Mohamed El Bachraoui recently considered partitions whose smallest part is repeated a specified number of times and any other parts are distinct \cite{AEB25}.

\begin{definition}
Given an integer $k \ge 1$, let $S_k(n) \subset P(n)$ be the partitions of $n$ where the smallest part occurs exactly $k$ times and any other parts are distinct.
\end{definition}

For example, $S_2(8) = \{(6,1,1), (4,4), (4,2,2))\}$.  Note that the $k=1$ case means that all parts appear exactly once, so that $S_1(n) = Q(n)$.  (Andrews and El Bachraoui use the notation $\text{Spt}k_d(n)$.)

The primary motivation for the current work concerns the following results from their section ``applications to partitions'' \cite[\S 3]{AEB25}.

\begin{proposition}[Corollaries 5 and 6, Andrews and El Bachraoui] \label{s3}
For any integer $n \ge 2$, 
\begin{equation}
s_2(n) = 2q(n-1) - q(n). \label{cor5a}
\end{equation}
For any integer $n \ge 4$,
\begin{equation}
s_3(n) = 2q(n-3) - 2q(n-1) + q(n) \label{cor5b}
\end{equation}
and
\begin{equation}
2q(n-1) - 2q(n-3) \le q(n) \le 2q(n-1). \label{cor6}
\end{equation}
\end{proposition}

Notice that the two inequalities of \eqref{cor6} follow from the quantities $s_3(n)$ and $s_2(n)$ being nonnegative in \eqref{cor5b} and \eqref{cor5a}, respectively.  

Proposition \ref{s3} includes results for which Andrews and El Bachraoui requested combinatorial proofs.  In the next section, we supply combinatorial proofs for \eqref{cor5a} and a new result, Theorem \ref{genSn}, from which identities analogous to \eqref{cor5b} follow for every $k$, not just $k = 3$.  

The family of identities in Theorem \ref{genSn} below also gives a sequence of inequalities expanding \eqref{cor6}.  In Section 3, we consider partitions where only the largest part can be repeated which will lead to another sequence of inequalities for $q(n)$.

\section{Results for all $s_k(n)$}

We begin with a combinatorial proof of \cite[Cor. 5(a)]{AEB25} as requested by Andrews and El Bachraoui to complement their analytic approach.

\begin{proof}[Combinatorial proof of \eqref{cor5a}]
Let $A(n) \subset S_2(n) \cup Q(n)$ be the partitions in the union with last part 1 and let $B(n)$ be the complementary subset of $S_2(n) \cup Q(n)$ with last part greater than one.  We establish bijections
\[ A(n) \cong Q(n-1), \;  B(n) \cong Q(n-1)\]
from which the claim follows.

For the first bijection, given $(\lambda_1, \ldots, \lambda_{t-1}, 1) \in A(n)$, remove the final part 1.  Whether the partition comes from $S_2(n)$ (with $\lambda_{t-1} = 1$ also) or from $Q(n)$ (with $\lambda_{t-1} > 1$), the image $(\lambda_1, \ldots, \lambda_{t-1}) \in Q(n-1)$.  For the reverse direction, given $(\mu_1, \ldots, \mu_s) \in Q(n-1)$, append a part 1, making the image in $A(n)$ (more specifically, the image is in $S_2(n)$ if $\mu_s = 1$, otherwise the image is in $Q(n)$).

For the second bijection, given $(\lambda_1, \ldots, \lambda_{t-1}, \lambda_t) \in B(n)$, replace the last part by $\lambda_t - 1$ (which is positive).  Whether the partition comes from $S_2(n)$ (with $\lambda_{t-1} = \lambda_t$) or from $Q(n)$ (with $\lambda_{t-1} > \lambda_t$), the image $(\lambda_1, \ldots,  \lambda_{t-1}, \lambda_t - 1) \in Q(n-1)$.  For the reverse direction, given $(\mu_1, \ldots, \mu_{s-1}, \mu_s) \in Q(n-1)$, increase the last part by one, making the image in $B(n)$ (the image is in $S_2(n)$ if $\mu_{s-1} = \mu_s + 1$, otherwise the image is in $Q(n)$).

Since $A(n) \cap B(n) = \varnothing$ and $A(n) \cup B(n) = S_2(n) \cup Q(n)$, the bijections establish $s_2(n) + q(n) = 2q(n-1)$, equivalent to \eqref{cor5a}.
\end{proof}

This proof also establishes that $A(n)$ and $B(n)$ are equinumerous subsets of $S_2(n) \cup Q(n)$, the partitions of $n$ with smallest part occurring once or twice and any other parts distinct \cite[A087135]{o}.

See Table \ref{5aex} for an example of the bijection.  

\begin{table}[h]
\centering
\caption{The $n=9$ example of the bijections in our proof of \eqref{cor5a}. Short horizontal bars separate partitions in $S_2(9)$ and $Q(9)$ within $A(9)$ and $B(9)$.} \label{5aex}
\renewcommand{\arraystretch}{1.25}
\begin{tabular}{l|l||l|l}
$A(9)$ & $Q(8)$ & $B(9)$ & $Q(8)$ \\ \hline
$(7,1,1)$		& $(7,1)$		& $(5,2,2)$	& $(5,2,1)$ \\ \cline{3-3}
$(5,2,1,1)$	& $(5,2,1)$	& $(9)$		& $(8)$ \\
$(4,3,1,1)$	& $(4,3,1)$	& $(7,2)$		& $(7,1)$ \\ \cline{1-1}
$(8,1)$		& $(8)$		& $(6,3)$		& $(6,2)$ \\
$(6,2,1)$		& $(6,2)$		& $(5,4)$		& $(5,3)$ \\
$(5,3,1)$		& $(5,3)$		& $(4,3,2)$	& $(4,3,1)$
\end{tabular}
\end{table}

Rather than give a combinatorial proof of \eqref{cor5b}, an identity involving $s_3(n)$, we establish the following general result.

\begin{theorem} \label{genSn}
Given integers $k \ge 2$ and $n \ge k$,
\[s_k(n) = - s_{k-1}(n) + s_{k-1}(n-k+1) + q(n-k+1).\]
\end{theorem}

\begin{proof}
Let $C(n) \subset S_k(n) \cup S_{k-1}(n)$ be the partitions in the union with last part 1 and let $D(n)$ be the complementary subset of $S_k(n) \cup S_{k-1}(n)$ with last part greater than one.  We establish bijections
\[ C(n) \cong Q(n-k+1), \; D(n) \cong S_{k-1}(n-k+1)\]
from which the claim follows.

For the first bijection, given $(\lambda_1, \ldots, \lambda_{t-k}, 1, \ldots, 1) \in C(n)$ (with $k-1$ parts 1 specified), remove the final $k-1$ parts 1.  Whether the partition comes from $S_k(n)$ (with $\lambda_{t-k} = 1$ also) or from $S_{k-1}(n)$ (with $\lambda_{t-k} > 1$), the image $(\lambda_1, \ldots, \lambda_{t-k}) \in Q(n-k+1)$.  For the reverse direction, given $(\mu_1, \ldots, \mu_s) \in Q(n-k+1)$, append $k-1$ parts 1, making the image $(\mu_1, \ldots, \mu_s, 1, \ldots, 1) \in C(n)$ (more specifically, the image is in $S_k(n)$ if $\mu_s = 1$, otherwise the image is in $S_{k-1}(n)$).

For the second bijection, given $(\lambda_1, \ldots, \lambda_{t-k}, i, \ldots, i) \in D(n)$ (with $k-1$ smallest parts $i > 1$ specified), decrease the final $k-1$ parts by one (note that $i-1$ is positive).  Whether the partitions comes from $S_k(n)$ (with $\lambda_{t-k} = i$ also) or from $S_{k-1}(n)$ (with $\lambda_{t-k} > i$), the image $(\lambda_1, \ldots, \lambda_{t-k}, i-1, \ldots, i-1) \in S_{k-1}(n-k+1)$.  For the reverse direction, given $(\mu_1, \ldots, \mu_{s-k}, j, \ldots, j) \in S_{k-1}(n-k+1)$ (with exactly $k-1$ parts $j$), increase the final $k-1$ parts by one, making the image $(\mu_1, \ldots, \mu_{s-k}, j+1, \ldots, j+1) \in D(n)$ (more specifically, the image is in $S_k(n)$ if $\mu_{s-k} = j+1$, otherwise the image is in $S_{k-1}(n)$).

Since $C(n) \cap D(n) = \varnothing$ and $C(n) \cup D(n) = S_k(n) \cup S_{k-1}(n)$, the bijections establish $s_k(n) + s_{k-1}(n) = q(n-k+1) + s_{k-1}(n-k-1)$, equivalent to the claim.
\end{proof}

See Table \ref{genSnex} for an example of the bijection.

\begin{table}[h]
\centering
\caption{The $n=10$, $k = 3$ example of the bijections used in the proof of Theorem \ref{genSn}. Short horizontal bars separate partitions in $S_3(10)$ and $S_2(10)$ within $C(10)$ and $D(10)$.} \renewcommand{\arraystretch}{1.25}
\begin{tabular}{l|l||l|l}
$C(10)$ & $Q(8)$ & $D(10)$ & $S_2(8)$ \\ \hline
$(7,1,1,1)$	& $(7,1)$		& $(4,2,2,2)$		& $(4,2,1,1)$ \\ \cline{3-3}
$(5,2,1,1,1)$	& $(5,2,1)$	& $(6,2,2)$		& $(6,1,1)$ \\
$(4,3,1,1,1)$	& $(4,3,1)$	& $(4,3,3)$		& $(4,2,2)$ \\ \cline{1-1}
$(8,1,1)$		& $(8)$		\\
$(6,2,1,1)$	& $(6,2)$		\\
$(5,3,1,1)$	& $(5,3)$		
\end{tabular}
\label{genSnex}
\end{table}

The $k = 2$ case of Theorem \ref{genSn} reduces to \eqref{cor5a}.  Substituting that into the $k = 3$ case for $n \ge 4$ gives
\begin{align*}
s_3(n) &= -s_2(n) + s_2(n-2) + q(n-2) \\
& = -(2q(n-1) - q(n)) + (2q(n-3) - q(n-2)) + q(n-2) \\
& = 2q(n-3) - 2q(n-1) + q(n)
\end{align*}
which is \eqref{cor5b}, i.e., we have an alternative proof of \cite[Cor. 5(b)]{AEB25} using our combinatorial proof of Theorem \ref{genSn}.

Continuing, the next two $s_k(n)$ expressions written in terms of the $q(n)$ sequence are
\begin{align}
s_4(n) & = 2q(n-6) - 2q(n-4) + 2q(n-1) - q(n), \label{s45} \\
s_5(n) & = 2q(n-10) - 2q(n-8) + 2q(n-6) - 2q(n-5) \notag \\
& \qquad {} + 2q(n-4) - 2q(n-1) + q(n). \notag
\end{align}

In the same way that the expressions for $s_2(n)$ and $s_3(n)$ give the two inequalities of \eqref{cor6}, Theorem \ref{genSn} gives an infinite family of inequalities above and below $q(n)$.  The following corollary records the next two results, which follow from the expressions for $s_4(n)$ and $s_5(n)$ in \eqref{s45}, respectively.  

\begin{corollary} \label{qS45}
For integers $n \ge 7$, 
\[ q(n)  \le 2q(n-1) - 2q(n-4) + 2q(n-6),\]
and, for $n \ge 11$, 
\[q(n) \ge 2q(n-1) - 2q(n-4) + 2q(n-5) - 2q(n-6) + 2q(n-8) - 2q(n-10).\]
\end{corollary}

There is a similar inequality for every value of $k$ since $s_k(n)$ has a combinatorial interpretation and is therefore nonnegative.

We mention, without further details, that the expressions for $s_k(n)$ in terms of the $q(n)$ sequence are equivalent to polynomials $P_k(q)$ used by Andrews and El Bachraoui in generating functions for $s_k(n)$.  Indeed, Theorem \ref{genSn} essentially gives a combinatorial proof of a recursive description of these polynomials \cite[Thm. 1]{AEB25}.

\section{The largest part}

Although considering the sum of the largest parts of partitions of $n$ does not seem to be as fruitful a concept as Andrew's $\text{spt}(n)$ \cite{A08}, the partitions analogous to $S_k(n)$ studied above for the largest part do lead to interesting results.

\begin{definition}
Given an integer $k \ge 1$, let $L_k(n) \subset P(n)$ be the partitions of $n$ where the largest part occurs exactly $k$ times and any other parts are distinct.
\end{definition}

For example, $L_2(8) = \{(4,4),(3,3,2)\}$ so that $\ell_2(8) = 2$.  Note that the $k=1$ case again means that all parts appear exactly once, thus $L_1(n) = Q(n)$.

Similar to Theorem \ref{genSn}, we establish the following result.

\begin{theorem} \label{genLs}
For integers $k \ge 2$ and $n \ge 1$, 
\[\ell_k(n) = -\ell_{k-1}(n) + \ell_{k-1}(n+k-1).\]
\end{theorem}

\begin{proof}
We establish a bijection
\[L_k(n) \cup L_{k-1}(n) \cong L_{k-1}(n+k-1)\]
from which the identity follows.

Given a partition in $L_k(n) \cup L_{k-1}(n)$, increase the first $k-1$ parts (all copies of the largest part) by one to make $\lambda \in L_{k-1}(n+k-1)$.  Note that partitions from $L_k(n)$ result in partitions with $\lambda_{k-1} - 1 = \lambda_k$ while partitions from $L_{k-1}(n)$ result in partitions with $\lambda_{k-1} - 1 > \lambda_k$.

For the reverse map, given $\lambda \in L_{k-1}(n+k-1)$, decrease the first $k-1$ parts by one to make a partition of $n$.  If $\lambda_{k-1}-1 = \lambda_k$, then the image is in $L_k(n)$.  If $\lambda_{k-1}-1 > \lambda_k$, then the image is in $L_{k-1}(n)$.

The maps are clearly inverses and establish $\ell_k(n) + \ell_{k-1}(n) = \ell_{k-1}(n+k-1)$, equivalent to the claim.
\end{proof}

See Table \ref{tabLs} for an example of the bijection.

\begin{table}[h]
\centering
\caption{The $n = 12$, $k = 3$ example of the bijection in the proof of Theorem \ref{genLs}.} \label{tabLs}
\renewcommand{\arraystretch}{1.25}
\begin{tabular}{l|l}
$L_3(12) \cup L_{2}(12)$ & $L_{2}(14)$ \\ \hline
$(4, 4, 4)$ & $(5,5,4)$ \\
$(3,3,3,2,1)$ & $(4,4,3,2,1)$ \\ \cline{1-1}
$(6,6)$ & $(7,7)$ \\
$(5,5,2)$ & $(6,6,2)$ \\
$(4,4,3,1)$ & $(5,5,3,1)$
\end{tabular}
\end{table}

In the same way that Theorem~\ref{genSn} can be used to express each $s_k(n)$ in terms of $q(n)$ and give inequalities such as those in \eqref{cor6} and Corollary~\ref{qS45}, we can use Theorem~\ref{genLs} to express each $\ell_k(n)$ in terms of $q(n)$ and derive corresponding inequalities.  The next corollary records those through $k = 5$.

\begin{corollary} \label{qL}
Each of the following statements is true for all $n \ge 1$.
\begin{enumerate}
\item[(a)] $\ell_2(n) = q(n+1) - q(n)$, thus
\[q(n) \le q(n+1); \]
\item[(b)] $\ell_3(n) = q(n+3) - q(n+2) - q(n+1) + q(n)$, thus
\[ q(n) \ge q(n+1) + q(n+2) - q(n+3);\]
\item[(c)] $\ell_4(n) = q(n+6) - q(n+5) - q(n+4) + q(n+2) + q(n+1) - q(n)$, thus
\[ q(n) \le q(n+1) + q(n+2) - q(n+4) - q(n+5) + q(n+6);\]
\item[(d)] $\ell_5(n) = q(n+10) - q(n+9) - q(n+8) + 2q(n+5) - q(n-2) - q(n-1) + q(n)$, thus
\[ q(n) \ge q(n+1) + q(n+2) - 2q(n+5) + q(n+8) + q(n+9) - q(n+10).\]
\end{enumerate}
\end{corollary}

Note that the inequalities of Corollary \ref{qL} can be rewritten in the more common form of descending arguments of $q(n)$; for example, the inequality (c) derived from $\ell_4(n)$ is equivalent to, for $n \ge 7$,
\[ q(n) \ge q(n-1) + q(n-2) - q(n-4) - q(n-5) + q(n-6).\]


\begin{thebibliography}{9}

\bibitem{A08}
G. E. Andrews, The number of smallest parts in the partitions of $n$, {\it J. Reine Angew. Math.} {\bf 624} (2008), 133--142.

\bibitem{AEB25}
G. E. Andrews and M. El Bachraoui, On the generating functions for partitions with repeated smallest part, {\it J. Math. Anal. Appl.} {\bf 549} (2025), 129537.

\bibitem{o}
OEIS Foundation Inc., The On-Line Encyclopedia of Integer Sequences, (2025), {\tt https://oeis.org}.

\end{thebibliography}
\end{document}